\documentclass[12pt]{amsart}
\usepackage{amsmath, amssymb, amsthm,  epsfig}

\theoremstyle{plain}
\newtheorem{thrm}{Theorem}[section]
\newtheorem*{thrm*}{Theorem}
\newtheorem{lemma}[thrm]{Lemma}
\newtheorem{prop}[thrm]{Proposition}
\newtheorem{cor}[thrm]{Corollary}
\theoremstyle{definition}
\newtheorem{dfn}[thrm]{Definition}
\theoremstyle{remark}
\newtheorem{rmrk}[thrm]{Remark}
\theoremstyle{example}

\numberwithin{equation}{section}
\usepackage{color}

\begin{document}

\newcommand{\tx}{\tilde x}
\newcommand{\R}{\mathbb R}
\newcommand{\N}{\mathbb N}
\newcommand{\C}{\mathbb C}
\newcommand{\lie}{\mathcal G}
\newcommand{\hN}{\mathcal N}
\newcommand{\D}{\mathcal D}
\newcommand{\A}{\mathcal A}
\newcommand{\B}{\mathcal B}
\newcommand{\sL}{\mathcal L}
\newcommand{\sLi}{\mathcal L_{\infty}}

\newcommand{\G}{\Gamma}
\newcommand{\x}{\xi}

\newcommand{\eps}{\epsilon}
\newcommand{\al}{\alpha}
\newcommand{\be}{\beta}
\newcommand{\p}{\partial}  
\newcommand{\lig}{\mathfrak}

\def\dist{\mathop{\varrho}\nolimits}

\newcommand{\BCH}{\operatorname{BCH}\nolimits}
\newcommand{\Lip}{\operatorname{Lip}\nolimits}
\newcommand{\Hol}{C}                             
\newcommand{\lip}{\operatorname{lip}\nolimits}
\newcommand{\capQ}{\operatorname{Cap}\nolimits_Q}
\newcommand{\pCap}{\operatorname{Cap}\nolimits_p}
\newcommand{\Om}{\Omega}
\newcommand{\om}{\omega}
\newcommand{\half}{\frac{1}{2}}
\newcommand{\e}{\epsilon}
\newcommand{\vn}{\vec{n}}
\newcommand{\X}{\Xi}
\newcommand{\tLip}{\tilde  Lip}

\newcommand{\Span}{\operatorname{span}}

\newcommand{\ad}{\operatorname{ad}}
\newcommand{\Hm}{\mathbb H^m}
\newcommand{\Hn}{\mathbb H^n}
\newcommand{\Hone}{\mathbb H^1}
\newcommand{\Lie}{\mathfrak}
\newcommand{\Layer}{V}
\newcommand{\hgrad}{\nabla_{\!H}}
\newcommand{\im}{\textbf{i}}
\newcommand{\nz}{\nabla_0}
\newcommand{\s}{\sigma}
\newcommand{\se}{\sigma_\e}

\newcommand{\ued}{u^{\e,\delta}}
\newcommand{\ueds}{u^{\e,\delta,\sigma}}
\newcommand{\tnabla}{\tilde{\nabla}}

\newcommand{\bx}{\bar x}
\newcommand{\by}{\bar y}
\newcommand{\bt}{\bar t}
\newcommand{\bs}{\bar s}
\newcommand{\bz}{\bar z}
\newcommand{\btau}{\bar \tau}

\newcommand{\LC}{\mbox{\boldmath $\nabla$}}
\newcommand{\Ne}{\mbox{\boldmath $n^\e$}}
\newcommand{\nuo}{\mbox{\boldmath $n^0$}}
\newcommand{\nuu}{\mbox{\boldmath $n^1$}}
\newcommand{\nue}{\mbox{\boldmath $n^\e$}}
\newcommand{\nuek}{\mbox{\boldmath $n^{\e_k}$}}
\newcommand{\dse}{\nabla^{H\Su, \e}}
\newcommand{\dso}{\nabla^{H\Su, 0}}
\newcommand{\tX}{\tilde X}

\newcommand\red{\textcolor{red}}
\newcommand\green{\textcolor{green}}

\newcommand{\Xie}{X^\epsilon_i}
\newcommand{\Xje}{X^\epsilon_j}
\newcommand{\Su}{\mathcal S}
\newcommand{\F}{\mathcal F}

\title[Mean curvature flow ]{Regularity of mean curvature flow \\ of graphs on Lie groups free up to step 2}

\author[Capogna]{Luca Capogna}
\address{Luca Capogna\\Department of Mathematical Sciences, Worcester Polytechnic Institute\\Worcester, MA 01609
}
\email{lcapogna@wpi.edu}

\author[Citti]{Giovanna Citti}
\address{Giovanna Citti\\ Dipartimento di Matematica, Universit\'a di Bologna \\
 Italy}
\email{giovanna.citti@unibo.it}
\author[Manfredini]{Maria  Manfredini}
\address{Maria  Manfredini\\ Dipartimento di Matematica, Universit\'a di Bologna \\
 Italy}\email{maria.manfredini@unibo.it}
\keywords{mean curvature flow, sub-Riemannian geometry, Carnot groups\\
{The authors are partially funded by MAnET "Metric Analysis for Emergent Technologies",
Marie Curie Initial Training Network, grant agreement n. 607643 (GC and MM) and by the NSF award DMS-1449143 (LC) }}

\begin{abstract} We consider (smooth) solutions of the mean curvature flow of  graphs over  bounded domains in a Lie group free up to step two (and not necessarily nilpotent), endowed with a one parameter family of Riemannian metrics $\sigma_\e$ collapsing to a subRiemannian metric $\sigma_0$ as $\e\to 0$. We establish $C^{k,\alpha}$ estimates for this flow, that are uniform as $\e\to 0$ and as a consequence prove long time existence for the subRiemannian mean curvature flow of the graph. Our proof extend to the setting of every step two Carnot group (not necessarily free) and can be adapted following our previous work in  \cite{CCM3} to the total variation flow.
\end{abstract}
\maketitle

\section{Introduction}

The mean curvature flow 
is the motion of a surface where each points 
is moving in the direction of the normal with speed equal to the mean curvature 
In the case where the evolution of graphs $S_t=\{(x,u(x,t))\}\subset \R^n\times \R$ is considered, 
then, provided enough regularity is assumed, the function $u$ satisfies the equation $$\p_t u=\sqrt{1+|\nabla u|^2}\, div\left(\frac{\nabla u}{\sqrt{1+|\nabla u|^2}}\right).$$ Given appropriate boundary/initial conditions, global in time solutions asymptotically converge to minimal graphs. 
 
 \bigskip
 In this paper we  study  long time existence of graph solutions of the mean curvature flow 
 in a special class of degenerate Riemannian ambient spaces: The so-called sub-Riemannian 
H\"ormander type setting \cite{fol:1975}, \cite{ste:harmonic}. In particular we will focus on a class of  Lie groups endowed with a metric structure  $(G,\sigma_0)$ that arises as limit of collapsing left-invariant {\it tame} Riemannian structures $(G,\sigma_\e)$.

 \bigskip

 Our approach to the existence of global (in time) smooth solutions is based on a Riemannian approximation scheme. 
We  study  graph solutions of the mean curvature flow
in the Riemannian spaces $(G,\sigma_\e)$ where $G$ is a  group
and $\sigma_\e$ is a family of  Riemannian metrics that 'collapse'  as $\e\to 0$ to a sub-Riemannian metric $\s_0$ in $G$. 

Our results are analogue to those we proved for the  total variation flow proved in \cite{CCM3}. 
The main difference is that we remove here the assumption that the group $G$ is a Carnot group, i.e. we also consider non-nilpotent groups such as the group of rigid Euclidean motions $\mathcal {RT}$. In fact the results in the present paper yield, with minor modifications of the proof, the regularity and long time existence of the total variation flow in the same extended class of groups. The main new technical  challenge that distinguishes the study of the total variation flow from the mean curvature flow  is that  the equation studied here is not in divergence form, which makes it necessary to  completely change the proof of the
$C^{1, \alpha}$ regularity.

\subsection{Lie  group structure}
Let $G$ be an analytic and simply connected Lie group with
topological dimension $n$.

A subRiemannian manifold on $G$ is a triplet $(G;\Delta; \sigma_0)$ where $\Delta$ denotes a left invariant bracket-generating subbundle of $TM$,  and $\sigma_0$ is a positive definite smooth, bilinear form on $\Delta$, see for instance
Montgomery
\cite{MONTGOMERY}. 
We fix a orthonormal horizontal basis  $X_{1},\ldots,X_{m}$ of $\Delta$.   
 We will say that the group has step 2 if $\{X_i\}_{i=1, \ldots, m} \cup \{ [X_i, X_j]\}_{i,j=1, \ldots, m}$ span the whole 
tangent space at every point. If in addition the vector fields 
$$\{X_i\}_{i=1, \ldots, m} \text{ and } \{[X_i, X_j]\}_{i,j=1, \ldots, m}$$ 
are linearly independent we say that the group is free up to step 2. 
We can complete $X_{1},\ldots,X_{m}$  to a basis $(X_1,\ldots,X_{n})$ of $\lie$, 
choosing a basis  of the second layer of the tangent space. We denote by $(X_1,\ldots,X_{n})$ (resp. $(X^r_1,\ldots,
X^r_{n})$) the left invariant (resp. right invariant) translations of
the frames $(X_1,\ldots,X_{n})$.
We will say that the vector fields $X_{1},\ldots,X_{m}$  have degree 1 and denote $d(X_i) =1$ 
while their commutators have  degree 2. Throughout the paper we will assume for simplicity that the horizontal frame above is self-adjoint.

As prototypes for this class of spaces we highlight the following:
\begin{itemize}
\item 
The standard example for such families is the Heisenberg group $\Hone$. This is a Lie group whose underlying manifold is $\R^3$ and is endowed with a group law $$(x_1,x_2,x_3)(y_1,y_2,y_3)=(x_1+y_1, x_2+y_2, x_3+y_3-(x_2y_1-x_1y_2)).$$ With respect to such law one has that the vector fields $$X_1=\p_{x_1}-x_2 \p_{x_3}\text{ and }X_2=\p_{x_2}+x_1\p_{x_3}$$ are left-invariant. Together with their commutator $[X_1,X_2]=2\p_{x_3}$ they yield a basis of $\R^3$. 
\item A second example is given by the classical group of rigid motions of the plane, also known as the {\it roto-translation} group $\mathcal {RT}$.  This is a Lie group with underlying manifold $\R^2\times S^1$
and a group law $(x_1,x_2,\theta_1)(y_1,y_2,\theta_2)=(x_1+y_1\cos \theta-y_2\sin\theta, x_2+y_1\sin \theta+y_2\cos\theta, \theta_1+\theta_2)$.  The horizontal distribution is given by 
$$\Delta=span\{ X_1,X_2\}, \text{ with }X_1=\cos\theta\p_{x_1}+\sin\theta \p_{x_2}, \text{ and } X_2=\p_\theta.$$ The subRiemannian metric $\sigma_0$ is defined so that $X_1$ and $X_2$ form a orthonormal basis. 
Note that $X_1$, $X_2$ and $[X_1, X_2] = X_3=-\sin\theta \p_{x_1}+\cos\theta \p_{x_2}$
span the tangent space at every point. \end{itemize}

The assumption that $\Delta $ is bracket generating,  allows to use the results from
\cite{nsw},  and define a control distance $d_0(x,y)$
 associated to the distribution $X_{1},\ldots, X_{m}$, which is
called {\sl the
 Carnot-Carath\'eodory metric} (denote by $d_{r,0}$
 the corresponding right invariant distance).  
  We let
$\nabla_0=(X_1,\ldots,X_m)$ denote the {\it horizontal gradient}
operator. If $\phi\in C^{\infty}(G)$ we set $\nabla_0 \phi=\sum_{i=1}^m X_i\phi X_i$ and $|\nabla_0 \phi|^2=\sum_{i=1}^m (X_i \phi )^2$.

We define a family of left invariant Riemannian metrics $\s_{\e}$,
$\e>0$ in $\lie$ by requesting that  $$\{X_1^\e,\ldots,X_n^\e\}:=\{X_1,\ldots, X_m, \e
X_{m+1}
 ,\ldots, \e X_n\}$$ is an orthonormal
frame. We will denote by $d_{\e}$ the corresponding distance
functions. Correspondingly we use $\nabla_\e,$ (resp. $
\nabla^r_\e$) to denote the left (resp. right) invariant gradients. In particular,
if  $\phi\in C^{\infty}(G)$ we set $\nabla_\e \phi=\sum_{i=1}^n X_i^\e\phi X_i^\e$ and  $|\nabla_\e \phi|^2=\sum_{i=1}^n (X_i^\e \phi )^2$.

\bigskip

Our results rest in a crucial way on the use of the celebrated Rothschild and Stein approximation theorem \cite{Roth:Stein}. Let  $X_1,\ldots,X_m$ be a bracket generating family of smooth vector fields, free up to step $2$, denote by $X_1,...,X_n$ its completion to a basis of of the tangent bundle $TG$ and set for every $u\in \R^n$,
\begin{equation}\label{phi0}
\Phi_{x}(u) = exp\Big(\sum_{i=1}^n   u_i  X_{i} \Big)(x).\end{equation}
We will use the following special case of the Rothschild-Stein osculating theorem \cite[Theorem 5]{Roth:Stein}, to approximate a neighborhood of any point in $G$ with a a neighborhood of the identity in the free group  $G_{m,2}$  with $m$ generators $Y_1,...,Y_m$ and step $2$. 

\begin{thrm}\label{rs-5}
Let  $X_1,\ldots,X_m$ be a family of smooth vector fields in $G$ that
are free up to rank $2$ at every point, as defined above.  Let  $G_{m,2}$ be the free Lie group of step two, with $m$ generators $Y_1,...,Y_m$ and set $Y_1,...,Y_n$ to be the  basis obtained by the original generators and their commutators.
For every $x\in G$ there exists a neighborhood $V$ of $x$ and a neighborhood $U$
of the identity   in $G_{m,2}$  such that:
\begin{enumerate}
\item[(a)] the map $\Phi_{x}:
U\to V $ 
is a diffeomorphism onto its image. We will denote by $\Theta_x$ its inverse map. Here we have denoted points in $U$ by their coordinates $(u_1,...,u_n)$ in the basis $Y_1,...,Y_n$.
\item[(b)]we have
\begin{equation}\label{split}
d \Theta_x (X_i)= Y_i+R_i, \ \ i=1,\ldots,m
\end{equation}
where $R_i$ is a vector field  of local degree less or equal than
zero, depending smoothly on $x.$
\end{enumerate}
In view of \cite[(14.5)]{Roth:Stein},  the operator $R_i$ is represented as
$$R_i= \sum_{ h=1}^n \sigma_{i h}(u) X_h,$$
where each  $\sigma_{i h}$ has a a Taylor expansion of homogeneous functions of degree larger or equal than    $d(X_h).$
\end{thrm} \index{Rothschild-Stein osculating theorem}

\subsection{The mean curvature flow}

If a surface $M$ is represented as a 0-level set of a function $f$, the points where the horizontal gradient of the defining function 
does not vanish are called non characteristic. At these points 
several equivalent definitions of the  horizontal mean curvature $h_0$ have been proposed. 
To quote a few: $h_0$ can be defined in terms of the first variation of the area functional
 \cite{CHMY, CHY, dgn:minimal, pau:cmc-carnot, RR1, cdpt:survey}, as horizontal divergence of the horizontal unit normal
 or as limit of the mean curvatures $h_\e$ of suitable Riemannian approximating {metrics} $\sigma_\e$
 \cite{cdpt:survey}.
If the surface is not regular, the notion of curvature can be expressed in the viscosity sense (we refer to  \cite{bieske}, \cite{bieske2},
\cite{wang:aronsson}, \cite{wang:convex}, \cite{lms},
\cite{baloghrickly}, \cite{magnani:convex},
\cite{CC} for viscosity solutions of PDE in the sub-Riemannian setting).

The mean curvature flow of a graph in $G\times \R$ is characterized by the fact that each point of the evolving  graph moves in the direction of the  {\it upward} unit normal with speed equal to the mean curvature. In the setting of the approximating Riemannian metrics $(G,\s_\e)$ this flow is smooth and in terms of the functions $t\to u_\e(\cdot,t):\Om\subset G \to \R$ describing the evolving graphs, the relevant equation reads:

\begin{equation}\label{pdee}\frac{\p u_{\e}}{\p t} =W_\e h_{\e}=W_\e\sum_{i=1}^n X_i^{\e}\Big(\frac{X_i^{\e}u_{\e}}{W_{\e} }\Big)=\sum_{i,j=1}^na_{ij}^\e (\nabla_\e u_\e) X_i^\e X_j^\e u_\e\quad \text{ for }x\in \Om, \; t>0, 
\end{equation}
where, $h_\e$ is the mean curvature of the graph of $u_\e (\cdot, t)$ and 
\begin{equation}\label{defaij} W_\e^2=1+|\nabla_\e u_\e|^2=
1+\sum_{i=1}^n (X_i^\e u_\e)^2 \text{ and } a_{ij}^\e(\xi)=  \delta_{ij}-\frac{\xi_i \xi_j}{1+|\xi|^2} ,
\end{equation}
 for all $\xi\in \R^n$.
In the sub-Riemannian limit $\e=0$ the equation reads
\begin{equation}\label{pde01}\frac{\p u}{\p t} =\sqrt{1+|\nabla_0 u|^2}\sum_{i=1}^m X_i  \Big(\frac{X_iu}{\sqrt{1+|\nabla_0 u|^2}}\Big),
\end{equation} with $u:=\lim_{\e\to 0} u_\e$, for $x\in \Om$ and $t>0$.
{Mean curvature flow in the setting of Carnot group has been studied by \cite{CC} and \cite{ccs}. See also the recent \cite{Dragoni} as well as \cite{Manfredi} for a probabilistic interpretation of the flow.}

\medskip
{\it The aim of this paper is to establish uniform (in the parameter $\e$ as $\e\to 0$)  estimates and determine  the asymptotic behavior of solutions
to the initial value problem for the mean curvature motion of graphs}
over bounded domains of a group $G$,
 \begin{equation}\label{ivp}
 \Bigg\{
 \begin{array}{ll}
 \p_t u_\e= h_\e W_\e  &\text{ in }Q=\Om\times(0,T) \\
 u_\e=\varphi &\text{ on } \p_p Q.
 \end{array}
 \end{equation}
 Here $\p_p Q=(\Om\times \{t=0\})\cup (\p\Om \times (0,T))$ denotes the parabolic boundary of $Q$.

The classical parabolic theory yields local existence and uniqueness for smooth solutions $u_\e$ of \eqref{ivp} of the problem, under suitable assumptions on the boundary data. Our main goal consists in proving that estimates are stable as $\epsilon$ tends to $0$, thus providing 
estimates also for the solution of the limit problem.
  
  Our first result consists in showing that if the initial/boundary data is sufficiently smooth then the solutions of \eqref{ivp} are Lipschitz up to the boundary uniformly in $\e>0$.
  
\begin{thrm}{(Global gradient bounds)}\label{global estimates}
Let $G$ be a Lie group of step two,
 $\Om\subset G$ a bounded, open, convex  set(in the sense of definition \ref{defconvex} below)  and $\varphi\in C^2(\bar{\Om})$. For $1\ge\e>0$ denote by  $u_\e\in C^2(\Om\times(0,T))\cap C^1(\bar\Om\times(0,T))$ the non-negative unique  solution of the initial value problem \eqref{ivp}. There exists $C=C(G, ||\varphi||_{C^2(\bar\Om)})>0$ such that
 \begin{equation}\label{b-1.1a}
  \sup_{\bar\Om\times(0,T)}|\nabla_1 u_\e| \le C.
 \end{equation}
 In particular, since $\sup_{\bar\Om\times(0,T)}|\nabla_\e u_\e| \le  \sup_{\bar\Om\times(0,T)}|\nabla_1 u_\e| $ one has uniform Lipschitz bounds for $u_\e$.
 \end{thrm}

Having established Lipschitz bounds, the next step is to recognize that right  derivatives $X^r_i u_\e$ of the solutions of \eqref{ivp} are solutions of  a divergence form, degenerate parabolic PDE. We prove that   weak solutions of such PDE satisfy a Harnack inequality and consequently  obtain $C^{1,\alpha}$ interior estimates for the original solution $u_\e$, which are uniform in $\e>0$. 

At this point one rewrites the PDE in \eqref{ivp} in  non-divergence form and invokes the {\it stable}
Schauder estimates for subriemannian equations (see \cite{CCM3}, \cite{CCstein}) to prove local higher regularity and long time existence.

\begin{thrm} \label{global in time  existence results}
In the hypothesis of Theorem \ref{global estimates}  there exists a unique solution
$u_\e \in C^{\infty}(\Om\times (0,\infty))\cap L^\infty((0,\infty),C^1(\bar \Om))$ of the initial value problem
\begin{equation}\label{ivp-inf}
 \Bigg\{
 \begin{array}{ll}
 \p_t u_\e= W_\e h_\e    &\text{ in }Q=\Om\times(0,\infty) \\
 u_\e=\varphi &\text{ on } \p_p Q 
 \end{array}
 \end{equation}
and that for each $k\in \N$ there exists $C_k=C_k(G,\varphi,k,\Om)>0$ not depending on $\e$ such that
\begin{equation}\label{stable estimates}
||u_\e||_{C^k(Q)} \le C_k.
\end{equation}
\end{thrm}

\bigskip

Since the estimates are uniform in $\epsilon$ and in time,
and with respect to $\epsilon$, we will deduce the following corollary:

\begin{cor} 
Under the assumptions of the Theorem \ref{global estimates},
as $\e\to 0$ 
the solutions $u_\e$ converge uniformly (with all its derivatives) on compact subsets of $Q$ to the unique,  smooth solution 
$u_0\in C^{\infty}(\Om\times (0,\infty))\cap L^\infty((0,\infty),C^1(\bar \Om))$ of the sub-Riemannian mean curvature flow  \eqref{pde01} in $\Om\times (0,\infty)$ with initial data $\varphi$.
\end{cor}

\begin{cor} 
Under the assumptions of Theorem \ref{global estimates},
as $T\to \infty$
the solutions $u_\e(\cdot, t)$ converge uniformly on compact subsets of $\Om$ to the unique  solution
of the minimal surface equation $$h_\e=0\quad in \; \Om$$ with boundary value $\varphi$,
while $u_0=\lim_{\e\to 0} u_\e\in C^\infty (\Om)\cap Lip(\bar \Om)$ is  the unique 
solution of the
{\it sub-Riemannian minimal surfaces} equation
$h_0=0$  in $\Om$, with boundary data $\varphi$.
\end{cor}

Regularity of minimal surfaces in the special case of Heisenberg group 
has  been investigated in \cite {GarofaloNhieu, Pauls, CHMY, CHY, CCM1, CCM2, DGNP, MR, RR1,  serracassano}.

%
%


\section{Structure stability in the Riemannian limit}

If $x\in G$ and $r>0$, we will denote by
 $$B(x,r)=\{y\in G \ |\ d_0(x,y)<r\}$$ the   balls
in the Carnot-Carath\'eodory control  distance corresponding to the subRiemannian metric $\sigma_0$. For each $\e>0$ we also define the distance function $d_\e$ corresponding to the Riemannian metric $\sigma_\e$,
$$d_\e(x,y)=\inf\{ \int_0^1 |\gamma'|_{\sigma_\e}(s) ds \text{ with }\gamma:[0,1]\to G \quad\quad\quad\quad\quad\quad\quad\quad $$$$\quad\quad\quad\quad\quad\quad\quad\quad\quad\quad\quad\quad
\text { a Lipschitz curve s. t. } \gamma(0)=x, \gamma(1)=y\}.$$ Set $$B_\e(x,r)=\{y\in G| d_\e(x,y)<r\}.$$
Note that  in the definition of $d_\e$, if the curve for which the infimum is achieved happens to be horizontal
then $d_\e(x,y)=d_0(x,y)$. In general we have $\sup_{\e>0} d_\e(x,y)= d_0(x,y)$ and
it is well known  that $(G,d_{\e})$  converges
in the Gromov-Hausdorff sense as $\e\to 0$ to the sub-Riemannian
space $(G,d_0)$. (See for instance \cite{CCstein, gro:metric} and references therein). 
\subsection{Stability of the homogenous structure  as $\e\to 0$}
%
If we denote by $dx$ the Haar measure in $G$, and by $|\Om|$ the corresponding measure of a subset $\Om\subset G$, then Rea and two of the authors have  proved in \cite{CCstein,CCR}
that \begin{prop}\label{homog-stab}
There exist  constants $C,R>0 $ independent of $\epsilon$ such that for every $x\in G$ and $R>r>0$,
$$|B_\e(x,2r)|\le C |B_\e(x,r)|.$$
 \end{prop}
Having this property the spaces $(G,d_\e,dx)$ are called {\it  homogenous} with constant $C>0$ independent of $\e$ (see \cite{cw:1971}).

Let $\tau>0$ and consider the space $\tilde G = G\times (0,\tau)$ with its  product Lebesgue measure $dxdt$. In $\tilde G$ define the pseudo-distance function
\begin{equation} \label{defde}\tilde d_{\e}( (x,t), (y,s))= \max( d_\e(x,y), \sqrt{|t-s|} ).
\end{equation}
 Proposition \ref{homog-stab} tells us that
$(\tilde G, \tilde d_{\e},dxdt)$ is a homogeneous space with constant independent of $\e\ge 0$. 
Likewise, the Poincar\'e inequality holds for all $\e$ near zero, with constant independent of $\e$.

\subsection{Stability of Schauder estimates}

Let us recall uniform estimates in spaces of H\"older continuous functions 
for solutions of second order sub-elliptic differential equations in non divergence form 
$$L_{\e, A} u\equiv \p_t u- \sum_{i,j=1}^n  a^\e_{ij}(x,t) X_i^\e X_j^\e u=0 ,$$
in a cylinder $ Q=\Om\times (0,T)$ that are stable as $\e\to 0$. We will  assume ellipticity
$$\Lambda^{-1}|\eta|^2\le  a^\e_{ij}(x,t) \eta^i \eta^j \le \Lambda |\eta|^2,$$  with $\Lambda>0$, and for a.e. $(x,t)\in Q$, all $\eta\in \R^n$ and $\e\in [0,1]$.

Let us start with the definition of classes of H\"older continuous functions in this setting

\begin{dfn}\label{defholder}
Let $0<\alpha < 1$, ${Q}\subset\R^{n+1}$ and  $u$ be defined on
${Q}.$ We say that $u \in C_{\e,X}^{\alpha}({Q})$ if there exists a positive constant $M$ such that for
every $(x,t), (x_{0},t_0)\in {Q}$ 
\begin{equation}\label{e301}
   |u(x,t) - u(x_{0},t_0)| \le M \tilde d_{\e }^\alpha((x,t), (x_{0},t_0)).
\end{equation}
We put  
 $$||u||_{C_{\e,X}^{\alpha}({Q})}=\sup_{(x,t)\neq(x_{0},t_0)} \frac {|u(x,t) - u(x_{0},t_0)|}{\tilde d_{\e }^\alpha((x,t), (x_{0},t_0))}+ \sup_{Q} |u|.$$
 Iterating this definition, 
 if $k\geq 1$  we say that $u \in
C_{\e,X}^{k,\alpha}({Q})$  if for all $i=1,\ldots ,m$
 $X_i u \in C_{\e,X}^{k-1,\alpha}({Q})$.
 Where we have set $C^{0,\alpha}_{\e,X}({Q})=C^{\alpha}_{\e, X}({Q}).$ 
\end{dfn}

Internal Schauder estimates for these type of operators are well known. 
We recall  the results of  Capogna
and Han \cite{CapognaHan} for uniformly subelliptic operators, of Bramanti and Brandolini \cite{BramantiBrandolini}
for heat-type operators,  and the results of Lunardi \cite{Lunardi} and 
Guti\'errez and Lanconelli  \cite{GutierrezLanconelli}, which apply to a large 
class of squares of vector fields plus a drift term.    
Schauder estimates uniform in $\e$ have been proved by the authors in \cite{CCM3} in the setting of Carnot Groups and by  two of us in \cite{CCstein} in the setting of general H\"ormander type vector fields. 

These result can be stated as

\begin{prop}\label{ellek}
Let  $w$ be a smooth  solution of $L_{\e, A}w=f$ on ${Q}$.
Let $K$ be a compact sets such that  $K\subset\subset {Q}$,  set $2\delta=d_0(K, \p_p Q)$ and
denote by $K_\delta$ the $\delta-$tubular neighborhood of $K$. Assume that 
there exists a constant $C>0$ such that 
$$ || a_{ij}^\e||_{C^{k,\alpha}_{\e,X}(K_\delta)} \leq C,$$ for any $\e\in (0,1)$.
There exists a constant $C_1>0$ depending  on
$\alpha$, $C$, $\delta$
but independent of $\e$,  such that
$$||w||_{C^{k+2, \alpha}_{\e,X}(K)} \leq C_1 \left( ||f||_{C^{k,\alpha}_{\e,X}(K_\delta)}+ ||w||_{C^{k+1, \alpha}_{\e,X}(K_\delta)}\right). $$
\end{prop}

\section{Gradient estimates}

In this section we prove Theorem \ref{global estimates}. The proof is carried out in two steps: First we use the maximum principle to establish interior $L^\infty$ bounds for the full gradient of the solution $\nabla_1 u$ of \eqref{ivp} with respect to the Lipschitz norm of $u$ on the parabolic boundary. Next, we construct appropriate barriers and invoke the comparison principle established in \cite{CC} to prove boundary gradient estimates. The combination of the two will yield the uniform global Lipschitz bounds.

\subsection{Interior gradient estimates} Recalling that the right invariant vector fields $X_j^r$ commute with the left invariant frame $X_i$, $i=1,\ldots ,n$ it is easy to show through a direct computation the following result.

\begin{lemma}\label{derivatedestre}
Let $u_\e\in C^3(Q)$ be a solution to \eqref{pdee} and denote
 $v_0=\partial_t u_\e$, $v_i = X_i^{r}u_\e$ for $i=i, \ldots, n$. Then
 for every $h=0,\ldots ,n$ one has that $v_h$ is a solution of
\begin{equation}\label{diff-eq}\p_t v_h= X_i^\e ( a_{ij} X_jv_h )= a_{ij}^\e(\nabla_\e u_\e) X_i^\e X_j^\e v_h + \p_{\xi_k} a_{ij}^\e(\nabla_\e u)X_i^\e X_j^\e u_\e X_k^\e v_h,\end{equation}
where
$$ a_{ij}^\e(\xi) =  \delta_{ij}- \frac{\xi_i \xi_j}{1+|\xi|^2}.$$
\end{lemma}

Note that in order to prove $L^\infty$ bounds on the horizontal gradient of solutions of \eqref{ivp} one {\it cannot} invoke Lemma \ref{derivatedestre} with differentiationalong  the horizontal left invariant  frame, because
such vector fields do not commute.

Let us explicitly note that the right derivatives 
$X^{1, r}_k$ are a basis of the tangent space, as well as 
$X^{1}_k$, so that it is possible to represent 
each family of vector fields as linear combination of the other. 
In particular, in the Carnot setting, it has been proved by \cite{Roth:Stein}
that there exist  homogenous polynomials $c_{kj}$ such that 
\begin{equation}\label{lr}X^{1}_k = \sum_j c_{kj} X^{1, r}_j.\end{equation}
In the general Lie groups setting this assertion is true only locally  
and  the functions  $c_{kj}$ are polynomials in the 
local exponential variables independent of $\e$. 

In view of this observation  and from the weak maximum principle one may easily deduce that:

\begin{prop}\label{u_t}
Let $u_\e\in C^3(Q)$ be a solution to \eqref{ivp} with $\Om$ bounded. There exists $C=C(G,||\varphi||_{C^2( \Om )})>0$
such that for every compact subset $K\subset \subset \Om$ one has
$$\sup_{K \times [0,T)} |\nabla_1 u_\e|  \leq \sup_{\partial_p Q}(|\nabla_1 u_\e| + |\partial _ t u_\e|),$$
where $\nabla_1$ is the  full $\s_1-$Riemannian gradient.
\end{prop}

\subsection{Barrier functions and boundary gradient estimate}

In \cite[Section 4.2]{CC} it is shown that in a step two Carnot group coordinate planes (i.e. images under the exponential of level sets of the form $x_k=0$) solve the minimal surface equation $h_0=0$. In the same paper
it is also shown that this may  fail for step three or higher. In order to adapt the construction of the barrier to the present non-nilpotent setting we will need a  refinement of this result, based on the following definition of convex set:
 \begin{dfn} \label{defconvex}
For every point $x_0\in \partial \Omega$  consider the canonical coordinates  
$\Phi_{x_0}(u)$ defined in (\ref{phi0}) and centered at $x_0$ (so that $x_0$ is represented by the origin in these coordinates). 
Assume that $\Omega$ has a tangent plane $\Pi$ 
at the point $u=0$ and assume that $\Phi_{x_0}^{-1}(\Omega)$ is lying 
on one side  of the plane. 
If this happens for every $x_0\in \partial\Omega$ we say that $\Omega$ is convex. 
 \end{dfn}
In a step two Carnot group this definition is equivalent to the one in \cite{CCM3}, and every set that is Euclidean  convex when expressed in exponential coordinates satisfies the condition. Using Darboux coordinates one can see that the same holds for the root-translation group $\mathcal{RT}$.
\begin{lemma}\label{planes are flat}
Let $G$ be a step two Carnot group. If $f:G\to \R$ is linear (in exponential coordinates) then for every $\e\ge 0$,  the matrix with entries
$X_i^\e X_j^\e f$ is anti-symmetric, in particular every level set of $f$ satisfies $h_\e=0$.
\end{lemma}

The previous lemma and the Rothschild and Stein local osculation result will lead 
to the construction of a barrier in the present setting, and to establish
a priori Lipschitz estimates at the boundary for  solutions. 
We begin by recalling an immediate consequence of the proof of \cite[Theorem 3.3]{CC}.

\begin{lemma}\label{ESth32bis}
For each $\e\ge 0$,
if  $v_\e$ is a bounded  subsolution and $w_\e$ is a bounded
 supersolution of \eqref{ivp} then $v_\e(x,t)\leq w_\e(x,t)$ for all $(x,t)\in Q$.
\end{lemma}

Let $u_\e\in C^2(Q)$ be a solution of \eqref{ivp}, and express the evolution PDE in non-divergence form
\begin{equation}\label{boundary 1.1}
\p_t u_\e =h_\e W_\e = a_{ij}^\e (\nabla_\e u) X_i^\e X_j^\e u_\e.
\end{equation}
Set $v_\e=u_\e-\varphi$ so that $v_\e$ solves the homogenous 'boundary' value problem
 \begin{equation}\label{ivph}
 \Bigg\{
 \begin{array}{ll}
 \p_t v_\e=  a_{ij}^\e (\nabla_\e v_\e+\nabla_\e \varphi) X_i^\e X_j^\e v_\e   +b^\e &\text{ in }Q=\Om\times(0,T) \\
 v_\e=0 &\text{ on } \p_p Q,
 \end{array}
 \end{equation}
with $b^\e(x)= a_{ij}^\e(\nabla_\e v_\e (x)+\nabla_\e \varphi(x) ) X_i^\e X_j^\e\varphi(x).$
We define our (weakly) parabolic operator for which the function $v_\e$ is  a solution
\begin{equation}\label{boundary 1.2}
Q(v)= a_{ij}^\e (\nabla_\e v_\e+\nabla_\e \varphi) X_i^\e X_j^\e v_\e   +b^\e -\p_t v.
\end{equation}

In the following we construct for each point $p_0=(x_0,t_0)\in \p\Om \times (0,T)$ a {\it barrier function} for $Q, v_\e$:  i.e.,

\begin{lemma}\label{barrier-lemma}
Let  $G$ be a Lie group free up to step two and  $\Om\subset G$  convex in the sense of Definition \ref{defconvex}. For  each point $p_0=(x_0,t_0)\in \p\Om \times (0,T)$ and for every $\e>0$ there exist a parabolic neighborhood $V_\e$ of $p_0$ and a positive function $w_\e\in C^2(Q)$ such that
\begin{equation}\label{boundary 1.3} Q(w_\e)\le 0 \text{ in }V_\e\cap Q 
\text{ with } w_\e\ge v_\e\text{ in }\p_pV_\e\cap Q.
\end{equation}
\end{lemma}

\begin{proof} For every $x_0\in \p \Om$ we can select exponential coordinates locally around the point $x_0$.
The point $x_0$ has coordinates $0$ in the variables $u$.

In these coordinates there exists an hyperplane $P$ tangent to the open set $\Omega$ defined by an equation of the form $\Pi(u)=\sum_{i=1}^n a_i u_i=0$ with $\Pi>0$ in $\Om$, $\Pi(0)=0$,
and normalized as $\sum_{d(i)=1,2}a_i^2=1$.
Following the standard argument (see for instance \cite[Chapter 10]{lieberman}) we select the barrier at $(x_0,t_0)\in \p \Om\times(0,T)$ independent of time with
\begin{equation}\label{boundary 1.4}
\tilde w_\e= \psi (\Pi)
\end{equation}
with $\psi$ solution of \begin{equation}\label{ode}\psi''+\nu (\psi')^2=0,
\end{equation}
 in particular
\begin{equation}\label{ode-sol}\psi(s)= \frac{1}{\nu} \log (1+ks),
\end{equation} with $k$ and $\nu$ chosen appropriately so that conditions
\eqref{boundary 1.3} will hold. We choose a neighborhood $V=O\times (0,T)$ such that $P\cap O\cap\p\Om=\{x_0\}$. By an appropriate choice of $k$ sufficiently large we can easily obtain $\tilde w_\e(0)=0 $ and $\tilde w_\e\circ \Theta_{x_0} \ge v_\e \text{ in }\p_pV\cap Q.$

We denote by  $Q_Y$ the operator which has the same expression of $Q$, but with respect to the left invariant osculating frame $\{Y_i^\e\}_{i=1,...,n}$ in the nilpotent osculating free group $G_{m,2}$, i.e. 
$Q_Y(v)= a_{ij}^\e (\nabla_{Y,\e} v_\e+\nabla_{Y,\e} \varphi) Y_i^\e Y_j^\e v_\e   +b^\e -\p_t v.$

To estimate $Q_Y(\tilde w_\e)\le 0$ we begin by observing that $\tilde w_\e$ satisfies
\begin{equation}\label{boundary 1.5}
Q_Y(\tilde w_\e)= \psi' a_{ij}^\e (\nabla_{Y, \e} \tilde w+\nabla_{Y,\e} \varphi) Y_i^\e Y_j^\e \Pi + 
\frac{\psi''}{(\psi')^2}\F +b_\e,
\end{equation}
with $\F= a_{ij}^\e (\nabla_{Y,\e} \tilde w_\e+\nabla_{Y, \e} (\varphi\circ \Phi_{x_0}) Y_i^\e \tilde w_\e Y_j^\e\tilde  w_\e.$

We first note that 

\begin{equation}\label{boundary 1.6}
     a_{ij}^\e \Big(\nabla_{Y, e} \tilde w_\e+\nabla_{Y, \e} (\varphi\circ \Phi_{x_0})\Big) (Y_i^\e Y_j^\e )\Pi  =0 \end{equation}
	as $ a_{ij}^\e$ is symmetric and
$X_i^\e X_j^\e \Pi$ is anti-symmetric in view of Lemma \ref{planes are flat}. 
We can now estimate the remaining terms of (\ref{boundary 1.5})
\begin{equation}
  \frac{\psi''}{(\psi')^2}\F +b_\e   \label{Claim 2}
\end{equation}

in a parabolic neighborhood of $u=0$.
We first note that  Lemma   \ref{planes are flat} implies  $$\frac{\e^2}{2}\le  \max(\sum_{d(i)=1} a_i^2, \e^2 \sum_{d(k)=1} a_k^2) \le |\nabla_{Y,\e} \Pi |=$$$$=\sum_{d(i)=1} \bigg( a_i + \sum_{d(k)=2, d(j)=1} c^k_{ij} a_k x_j\bigg)^2+ \e^2 \sum_{d(k)=2} a_k^2 \le C(G) (1 + \e^2), $$ for some constant $C(G)>0$.
Consequently, for $\psi'>>1$ sufficiently large one finds
\begin{equation}\label{boundary 1.8}
\F\ge \frac{|\nabla_{Y, \e} \tilde w_\e|^2}{1+|\nabla_{Y,\e} \tilde w_\e+\nabla_{Y,\e} \varphi|^2} 
\ge C(G) \frac{|\nabla_{Y,\e} \tilde w_\e|^2}{1+|\psi'|^2 +|\nabla_{Y,\e} \varphi|^2} \ge C(G) \e^2>0,
\end{equation}
with $C(G)>0$ a constant depending only on $G$ (not always the same along the chain of inequalities). In view of the definition of $b_\e$ and \eqref{ode} with an appropriate choice of $\nu=\nu(G,\e, \phi)>0$ and $k=k(G,\phi)>>1$ in \eqref{ode-sol}, we conclude
\begin{equation}\label{boundary 19}
 \frac{\Phi''}{(\Phi')^2}\F +b_\e \le \bigg(  \frac{\psi''}{(\psi')^2} +\nu-1\bigg) \F\leq - C(G) \e^2.
\end{equation}
It follows that 
$$Q_Y(\tilde w_\e)\leq - C(G) \e^2.$$
To conclude, we set
$w_\e=\tilde w_\e \circ \Theta _{x_0}$.
In view of  the relation \eqref{split} between the vector fields $X$ and $Y$,  it is now immediate to see that
there exists a neighborhood  $V_\e$ of $p_0$, depending  on $\e$, such that 
$$Q(w_\e)\leq 0 \text{ in }V_\e\cap Q. $$
\end{proof}


\begin{prop}\label{boundarygradientestimate}
Let  $G$ be a Lie group free up to step two, $\Om\subset G$  convex in the sense of Definition \ref{defconvex}
and $\varphi\in C^2(\bar\Om)$. For $\e>0$ denote by  $u_\e\in C^2(\Om\times(0,T))\cap C^1(\bar\Om\times (0,T))$ the non-negative unique  solution of the initial value problem \eqref{ivp}. There exists $C=C(G, ||\varphi||_{C^2(\bar\Om)})>0$ such that
 \begin{equation}\label{b-1.1}
 \sup_{\p\Om\times(0,T)}|\nabla_\e u_\e| \le  \sup_{\p\Om\times(0,T)}|\nabla_1 u_\e| \le C.
 \end{equation}

\end{prop}

\begin{proof}
In view of Lemma \ref{ESth32bis}, a comparison with the barrier constructed above yields that
\begin{equation}\label{boundary 20}
0\le \frac{v_\e (x,t)}{dist_{\sigma_1}(x,x_0)} \le  \frac{w_\e(x,t)}{{dist_{\sigma_1}(x,x_0)}} \le C(k,\nu),
\end{equation}
in $V_\e\cap Q$, with $dist_{\sigma_1}(x,x_0)$ being the distance between $x$ and $x_0$ in the Riemannian metric  $\sigma_1$,  concluding the proof of the boundary gradient estimates at the point $p_0$.
\end{proof}



\section{Regularity properties in the $C^{k, \alpha}$ spaces}

In this section we will prove uniform estimates for solution of \eqref{pdee}  in the $C^{k,\alpha}_{\e,X}$  H\"older spaces.  This is accomplished  in two steps and follows a strategy originally introduced by Trudinger (see notes in \cite[Chapter 7]{lieberman}. 
First we establish $C^{1, \alpha}$ regularity, 

\subsection{Regularity properties in the $C^{1, \alpha}$ spaces}

\begin{rmrk}
Let $u$ be a smooth solution of  (\ref{pdee}). Note that the function $v_k=X_k^r u$ is then a solution of the equation 
\begin{equation}\label{RG-1}
-\partial_t v_k +\sum_{i,j=1}^n X_i^\e \Big(a_{ij}^\e (\nabla _\e u) X_j^\e v_k\Big) 
+ \sum_{i,j,h=1}^n a^{i,j,h}  X_i^\e X_j^\e u_\e X_h^\e v_k=0,\end{equation}
 where $$a^{i,j,h}=\frac{\partial a^\e_{ij}}{\partial p_h}  
 - \frac{\partial a^\e_{ih}}{\partial p_j}$$
\end{rmrk}
Indeed taking the derivative of equation (\ref{pdee}) and taking into account that the right and left derivatives commute, one obtains
$$-\partial_t X^r_ku_\e +  \sum_{i,j,h=1}^n\frac{\partial a^\e_{ij}}{\partial p_h} (\nabla _\e u) X_i^\e X_j^\e u_\e  X_h^\e X^r_ku_\e+ \sum_{i,j=1}^n a_{ij}^\e (\nabla _\e u)  X_i^\e X_j^\e  X^r_ku_\e =0 $$
Consequently 
$$-\partial_t v_k + \sum_{i,j=1}^nX_i^\e \Big(a_{ij}^\e (\nabla _\e u) X_j^\e v_k\Big) 
+ \sum_{i,j,h=1}^n \frac{\partial a_{ij}^\e}{\partial p_h}  
X^\e_iX^\e_j u_\e X_h^\e v_k - \sum_{i,j,h=1}^n \frac{\partial a_{ij}^\e}{\partial p_h}X_j^\e X_k^r u_\e X_i^\e X_h^\e u_\e=0$$
and the latter is equivalent to \eqref{RG-1}.

\begin{rmrk}
Starting \eqref{RG-1} one can immediately see that the 
function  $z= |\nabla^r_{\e}u_\e|^2$ is solution of 
$$-\partial_t z + \sum_{i,j=1}^nX_i^\e \Big(a_{ij}^\e (\nabla _\e u) X_j^\e z\Big) + \sum_{i,j,h=1}^n a^{ijh}X_i^\e X_j^\e u_\e X_h^\e z 
- 2 \sum_{i,j,h,k=1}^n a_{ij}^\e X_i^\e v_k X^\e_j v_k =0$$
\end{rmrk}
%
%
%
%

\begin{lemma}\label{remarkstep1}
for every $k=1, \ldots, n$ and for every  $\delta>0$ 
the functions 
$$w_k^{\pm} = \pm v_k + \delta z$$
satisfy the inequality
$$-\partial_t w_k^{\pm}  + \sum_{i,j=1}^n X_i^\e \Big(a_{ij} ^\e(\nabla _\e u) X_j^\e w_k^{\pm}\Big)\geq - C_0 |\nabla_\e w_k^{\pm}|^2 - C_1,$$
for suitable constants $C_0$ and $C_1$. 
\end{lemma}

\begin{proof}
For simplicity we temporarily drop the $\pm$ superscript.
Adding the equations satisfied by $v_k$ and $z$ we see that for every $k=1, \ldots, n$ and for every  $\delta>0$ 
the functions 
$w_k^{\pm} = \pm v_k + \delta z$
satisfy 
\begin{equation}\label{RG-2} 
-\partial_t w_k + \sum_{i,j=1}^n X_i^\e \Big(a_{ij}^\e (\nabla _\e u) X_j^\e w_k\Big) = - \sum_{i,j,h=1}^n a^{ijh}X_i^\e X_j^\e u_\e X_h^\e w_k
+ 2 \delta \sum_{i,j,s=1}^n a^\e_{ij} X_i^\e v_s X^\e_j v_s \end{equation}
$$\geq 2 \delta \lambda \sum_{s=1}^n |\nabla_\e X_s^{1, r} u_\e|^2 - \sup_{i,j,h} |a^{ijh}| |X_i^\e X_j^\e u_\e| |\nabla_\e w_k|, $$
where $\lambda>0$ is the smallest eigenvalue of $a_{ij}^\e$ and is independent of $\e>0$.
Using the notation introduced in  (\ref{lr}) we deduce that 
$$|X_i^\e X_j^\e u_\e| = |\sum_{j,s=1}^n  X_i^\e (c_{sj}X_s^{1, r} u_\e)| \leq 
  C_2 + C_3\sum_{s=1}^n |\nabla_\e X_s^{1,r} u_\e|.$$
Consequently, using Schwarz's inequality,
$$-\partial_t w_k + \sum_{ij}X_i \Big(a_{ij}^\e (\nabla _\e u) X_j^\e w_k\Big)\geq$$
$$\geq 2 \delta \lambda  \sum_{s=1}^n |\nabla_\e X_s^{1, r} u_\e|^2 -  \sum_{s=1}^n \delta \lambda |\nabla_\e X_s^1 u_\e|^2  
- C_0 |\nabla_\e w_k|^2 - C_1, $$
completing the proof. \end{proof}

Next we set $( x_0, t_0)\in Q=\Om\times (0,T)$ and for $r>0$, let  $Q_\e(r)=\{(x,t)\in Q | d_\e(x,x_0)<r $ and $|t-t_0|\le r^2\}$. Define 
$$W^\pm_k = sup_{Q_\e(4r)} w^\pm_k$$ 
and observe that 
$$\p_t (W^\pm_k-w^\pm_k) - \sum_{i,j=1}^n X_i^\e \Big(a_{ij}^\e(\nabla _\e u) X_j^\e (W^\pm_k - w_k^{\pm})\Big)\geq - C_0 |\nabla_\e (W^\pm_k-w_k^{\pm})|^2 - C_1.$$ 
In order to invoke the weak Harnack inequality and derive the $C^\alpha$ estimates, we need to eliminate the quadratic term on the right hand side. Following \cite[Chapter 12, Sec. 3]{lieberman} we define   
$$\bar w_k = \frac{\lambda}{2C_0} \Big(1 - exp(\frac{2C_0}{\lambda}  (w^\pm_k - W^\pm_k))\Big)$$
and observe that this new functions satisfies 
$$-\partial_t \bar w_k  + \sum_{ij}X_i^\e \Big(a_{ij}^\e (\nabla _\e u) X_j^\e \bar w_k\Big) +  g \leq 0,$$
where $  g =C_1(\frac{2C_0}{\lambda} \bar w_k + 1)$, for the constants $\lambda, C_0, C_1$ from Lemma \ref{remarkstep1}.  In view of the weak Harnack inequality \cite[Proposition 7.6]{CCstein}, one has that for some constant $C_4>0$ independent of $\e$ and for $Q_\e^-(r) =\{(x,t)\in Q| \ d_\e(x,x_0)<r$ and $t_0-3r^2<t<t_0<2t^2\}$, 
$$\int_{Q_\e^-(r)} \bar w_k \ dx dt \le C_4(\inf_{Q_\e(r)} \bar w_k  +\sup_{Q_\e(r)} |g| r^2).$$
Following the argument  in \cite[Chapter 12, Sec. 3]{lieberman}
we obtain
\begin{prop} Let $u_\e$ be a solution of the mean curvature flow PDE \eqref{pdee} in $Q=\Om\times (0,T)\subset G\times \R$. 
Let $K$ be a compact sets such that  $K\subset\subset {Q}$,  set $2\delta=d_0(K, \p_p Q)$ and
denote by $K_\delta$ the $\delta-$tubular neighborhood of $K$ in $d_0$.
There exists  constants $C>0$  and $\alpha \in (0,1)$ depending  on
 $\delta$ and on the Lipschitz norm of $u$ in $K_\delta$,
but independent of $\e$,  such that
$$||u_\e||_{C^{1, \alpha}_{\e,X}(K)} \leq C . $$
\end{prop}

\subsection{Regularity properties in the $C^{k, \alpha}$ spaces}

Once obtained the  interior $C^{1, \alpha}$ estimate of the solution uniform in $\e$, we  
write the  mean curvature flow equation in non divergence form: 
$$\p_t u- \sum_{i,j=1}^n  a^\e_{ij}(x,t) X_i^\e X_j^\e u=0 ,$$
Applying Schauder estimates (see \cite{CCM3} for the Carnot groups setting and 
\cite{CCstein} for the general Lie group case). we immediately deduce the proof of Theorem 
\ref{global in time  existence results}. 

\begin{proof}
Since the solution is of class $C^{1, \alpha}$, and the norm is bounded uniformly in $\e$ 
then $u_\e$ it is a solution of a divergence form equation 
$$\p_t u_\e- \sum_{i,j=1}^n  a^\e_{ij}(x,t) X_i^\e X_j^\e u_\e=0 ,$$
with $a_{ij}^\e$ of class $C^\alpha$ such that for every 
 $K$ be a compact sets such that  $K\subset\subset {Q}$ and  $2\delta=d_0(K, \p_p Q)$
there exists a positive constant $C_0$  such that 
$$ || a_{ij}^\e||_{C^{ \alpha}_{\e,X}(K_\delta)} \leq C_0,$$ 
for every  $\e\in (0,1)$.
Consequently,  by Proposition \ref{ellek} there exists a constant $C_2$ such that 
$$||u_\e||_{C^2(Q)} \le C_2.$$
The conclusion immediately follows by induction. 
\end{proof}

\end{document}